\documentclass{amsart}
\usepackage{xy,tikz,float}
\xyoption{all}
\usepackage{MnSymbol}
\usepackage{comment}
\title{}
\author{A. Salch}
\title{Ravenel's May spectral sequence collapses immediately at large primes.}
\date{November 2023.}
\DeclareMathOperator{\strict}{{\rm str}}
\DeclareMathOperator{\Aut}{{\rm Aut}}
\DeclareMathOperator{\Gal}{{\rm Gal}}
\DeclareMathOperator{\Cotor}{{\rm Cotor}}
\newcommand{\floor}[1]{\lfloor {#1}\rfloor}
\DeclareMathOperator{\Spec}{{\rm Spec}}
\DeclareMathOperator{\Ext}{{\rm Ext}}
\DeclareMathOperator{\lowert}{{\it t}}
\DeclareMathOperator{\lowerb}{{\it b}}
\DeclareMathOperator{\Tor}{{\rm Tor}}
\DeclareMathOperator{\tensor}{\otimes}
\theoremstyle{plain}
\newtheorem{prop}{Proposition}[section]
\newtheorem{theorem}[prop]{Theorem}
\newtheorem{corollary}[prop]{Corollary}

\newtheorem{lemma}[prop]{Lemma}
\newtheorem{definition}[prop]{Definition}
\newtheorem{definition-proposition}[prop]{Definition-Proposition}
\newtheorem{definition-theorem}[prop]{Definition-Theorem}

\theoremstyle{definition}
\newtheorem{remark}[prop]{Remark}

\newtheorem{question}[prop]{Question}

\newtheorem{observation}[prop]{Observation}

\usepackage{dutchcal}
\usepackage{hyperref}
\usepackage{cleveref}
\begin{document}
\begin{abstract}
At large primes, the height $n$ Ravenel-May spectral sequence takes as input the cohomology of a certain solvable Lie $\mathbb{F}_p$-algebra, and produces as output the mod $p$ cohomology of the height $n$ strict Morava stabilizer group scheme. We construct simultaneous integral deformations of the height $n$ Morava stabilizer algebras and related objects, and we use them to prove that, for fixed $n$, the height $n$ Ravenel-May spectral sequence collapses for all sufficiently large primes $p$. Consequently, for large $p$, the mod $p$ cohomology of the strict Morava stabilizer group scheme is the cohomology of a finite-dimensional solvable Lie algebra, and is computable algorithmically.
\end{abstract}
\maketitle

\section{Introduction.}
\label{Introduction.}

Let $n$ be a positive integer. Let $p$ be a prime number satisfying $p>n+1$.
In the 1977 paper \cite{ravenel1977cohomology}, Ravenel used the methods of May's thesis \cite{MR2614527} to construct a spectral sequence
\begin{align}
\label{rmss 0} E_1^{s,t,u} \cong H^{s,t,u}(L(n,n);\mathbb{F}_p) &\Rightarrow H^{s,t}_c(\strict\Aut(\mathbb{G}_{1/n});\mathbb{F}_p),
\end{align}
which we call the {\em $p$-primary height $n$ Ravenel-May spectral sequence}, or RMSS for short.
Here $L(n,n)$ is a particular $n^2$-dimensional solvable Lie $\mathbb{F}_p$-algebra, and $\strict\Aut(\mathbb{G}_{1/n})$ is the height $n$ strict Morava stabilizer group scheme\footnote{Some readers may be much less familiar with the Morava stabilizer group {\em scheme} than with the Morava stabilizer {\em group}. We give some background about the Morava stabilizer group scheme in \cref{Review of Ravenel...}, but the upshot is the following three points: 
1. The Morava stabilizer group scheme is the prime spectrum of the Hopf algebra $S(n)$, the Morava stabilizer algebra. In this paper, we usually write $S(n,p)$ rather than the more common notation $S(n)$, to emphasize the role of the prime $p$. The cohomology of the Morava stabilizer group scheme is simply $\Cotor$ over $S(n,p)$. 
2. In every situation in stable homotopy in which one considers the cohomology of the Morava stabilizer group, it is at least as good to consider the cohomology of the Morava stabilizer group scheme---and sometimes it is even a bit better, in the sense that it sometimes allows us to skip a descent spectral sequence for a $\Gal(\mathbb{F}_{p^n}/\mathbb{F}_p)$-action. 3. The main results of this paper remain true after replacing the Morava stabilizer group scheme with the Morava stabilizer group, throughout.}, i.e., the strict automorphism group scheme of a height $n$ one-dimensional formal group law over $\mathbb{F}_p$. 
The $E_1$-page of the RMSS is the cohomology of the Lie algebra $L(n,n)$, while the abutment of the RMSS is the continuous cohomology of the profinite group scheme $\strict\Aut(\mathbb{G}_{1/n})$. The latter is the input for many spectral sequences which are used to calculate stable homotopy groups. For example:
\begin{itemize}
\item From $H^*_c(\strict\Aut(\mathbb{G}_{1/n});\mathbb{F}_p)$, one runs a sequence of $n$ Bockstein spectral sequences to get the height $n$ line in the chromatic spectral sequence $E_1$-page. The chromatic spectral sequence converges to the Adams-Novikov $E_2$-page, and the Adams-Novikov spectral sequence converges to the stable homotopy groups of spheres. See chapters 4-6 of \cite{MR860042} for this material.
\item From $H^*_c(\strict\Aut(\mathbb{G}_{1/n});\mathbb{F}_p)$, one runs a sequence of $n$ Bockstein spectral sequences (which are dual, in a certain sense, of the $n$ Bockstein spectral sequences mentioned just above) to get to $H^*_c(\strict\Aut(\mathbb{G}_{1/n});E(\mathbb{G}_{1/n})_*)$, where $E(\mathbb{G}_{1/n})_*$ is the Morava $E$-theory spectrum of $\mathbb{G}_{1/n}$. The continuous cohomology $H^*_c(\strict\Aut(\mathbb{G}_{1/n});E(\mathbb{G}_{1/n})_*)$ is the input for spectral sequences converging to $\pi_*(L_{K(n)}S^0)$, the $K(n)$-local stable homotopy groups of spheres. See \cite{MR2030586} for this material.
\item If the $p$-local Smith-Toda complex $V(n-1)$ exists, then $H^*_c(\strict\Aut(\mathbb{G}_{1/n});\mathbb{F}_p)\otimes_{\mathbb{F}_p}\mathbb{F}_p[v_n^{\pm 1}]$ is the input for a spectral sequence converging to the $K(n)$-local stable homotopy groups of $V(n-1)$. If $p > \frac{n^2+n+2}{2}$, then this spectral sequence collapses immediately with no differentials, yielding an isomorphism $H^*_c(\strict\Aut(\mathbb{G}_{1/n});\mathbb{F}_p)\otimes_{\mathbb{F}_p}\mathbb{F}_p[v_n^{\pm 1}]\cong \pi_*(L_{K(n)}V(n-1))$. See Corollary \ref{smith-toda cor} for further explanation.
\end{itemize}
Because of the applications of $H^*_c(\strict\Aut(\mathbb{G}_{1/n});\mathbb{F}_p)$ in computational stable homotopy theory, it is of great interest to know an answer to the following question:
\begin{question}\label{main q}
For which values of $p$ and $n$ does the RMSS collapse immediately?
\end{question}
When the RMSS collapses immediately, $H^*_c(\strict\Aut(\mathbb{G}_{1/n});\mathbb{F}_p)$ is simply the cohomology of a certain solvable finite-dimensional Lie algebra, computable (algorithmically!) via a Chevalley-Eilenberg complex \cite{MR0024908}.

Since the Ravenel-May spectral sequence only has the stated\footnote{When $p\leq n+1$, a Ravenel-May spectral sequence still exists, and still converges to $H^{*}_c(\strict\Aut(\mathbb{G}_{1/n});\mathbb{F}_p)$, but its $E_1$-page is more complicated than described above in \eqref{rmss 0}. Its $E_1$-page is the {\em restricted} cohomology of a certain {\em restricted} Lie algebra which is larger than $L(n,n)$. See Theorem 1.5 of \cite{ravenel1977cohomology} or Theorem 6.3.4 of \cite{MR860042}.} form \eqref{rmss 0} when $p>n+1$, the most optimistic possible answer for Question \ref{main q} would be: the Ravenel-May spectral sequence collapses immediately as long as $p>n+1$. Let us call this the {\em optimistic answer} to Question \ref{main q}.

Indeed, in Theorem 1.7 of \cite{ravenel1977cohomology}, Ravenel claims that this optimistic answer is true. However, Ravenel later pointed out \cite[section 6.3]{MR860042} that the argument given for the optimistic answer in \cite{ravenel1977cohomology} is incorrect. In \cref{Review of Ravenel...}, below, we review Ravenel's original argument for the optimistic answer. In \cite{MR860042}, Ravenel also describes a possibly nonzero differential in the $p=11,n=9$ case of the Ravenel-May spectral sequence which he suggests may yield a counterexample to the optimistic answer.

Question \ref{main q} remains wide open. In particular, it is today not even known whether the optimistic answer is true for $n>3$. From explicit calculation, it is known that the optimistic answer is indeed correct for $n=1,2$ \cite[section 6.3]{MR860042} and $n=3$ \cite[section 5]{MR1173994} (see also \cite{MR4301320} for a simplification of the argument from \cite{MR1173994}, at $p\geq 7$). In each of those cases, one verifies the correctness of the optimistic answer by explicitly calculating $H^*(L(n,n);\mathbb{F}_p)$, then checking that $H^*(L(n,n);\mathbb{F}_p)$ is generated as a ring by elements which, by routine inspection of tridegrees, cannot support a nonzero Ravenel-May differential. This strategy for answering Question \ref{main q} does not generalize well to arbitrary values of $n$, since the computational cost of calculating $H^*(L(n,n);\mathbb{F}_p)$ increases rapidly\footnote{We offer a little bit of data about the difficulty of calculating $H^*(L(n,n);\mathbb{F}_p)$. The brute-force approach is to simply calculate the cohomology of the Chevalley-Eilenberg complex $\Lambda^{\bullet}(L(n,n)^*)$ of $L(n,n)$. This is a finite calculation, and in principle a computer can attempt this, and it can succeed for low values of $n$. However, since $L(n,n)$ is $n^2$-dimensional, the Chevalley-Eilenberg complex is $2^{(n^2)}$-dimensional. This grows extremely quickly as $n$ grows, and exhausts any feasible supply of RAM if $n>5$.

The cohomology of $L(n,n)$ also grows quickly as $n$ grows. For $p>n+1$, the total $\mathbb{F}_p$-linear dimension of $H^*(L(n,n);\mathbb{F}_p)$ and of the Chevalley-Eilenberg complex is given by the following table:
\begin{equation*}
\begin{array}{lll}
\mbox{n}          & \dim_{\mathbb{F}_p}H^*(L(n,n); \mathbb{F}_p) & \dim_{\mathbb{F}_p}\Lambda^{\bullet}(L(n,n)^*) \\
1 & 2 & 2 \\
2 & 12 & 16 \\ 
3 & 152 & 512 \\
4 & 3440 & 65536 \end{array}\end{equation*}
The cohomology $H^*(L(n,n);\mathbb{F}_p)$ is not known for $n>4$. The $n=4$ case is in the preprint \cite{height4}, where it is also observed that these $\mathbb{F}_p$-linear dimensions agree with the coefficients in a particular generating function. If the coefficients of that generating function continue (for higher $n$) to agree with the $\mathbb{F}_p$-linear dimension of $H^*(L(n,n);\mathbb{F}_p)$, then, for example, $H^*(L(5,5);\mathbb{F}_p)$ will be $128512$-dimensional, while $H^*(L(6,6);\mathbb{F}_p)$ will be $7621888$-dimensional.

In the preprint \cite{height4}, a ``height-shifting'' approach is used to calculate $H^*(L(4,4);\mathbb{F}_p)$. This approach is much more efficient and illuminating than brute-force calculation of the cohomology of the Chevalley-Eilenberg complex, but it is still quite demanding, increasingly so as $n$ grows.} as $n$ grows.

In this paper, we give an {\em asymptotic} answer to Question \ref{main q}. The main results are:
\newtheorem*{unnumberedthm:main1}{Theorem \ref{main thm 1}}
\begin{unnumberedthm:main1}
Fix a positive integer $n$. Then there exists some integer $N_n$ such that, for all $p>N_n$, the $p$-primary height $n$ Ravenel-May spectral sequence has no nonzero differentials.
\end{unnumberedthm:main1}
\newtheorem*{unnumberedthm:main2}{Theorem \ref{main thm 2}}
\begin{unnumberedthm:main2}
Fix a positive integer $n$. Then there exists some integer $N$ such that, for all $p>N$, the mod $p$ continuous cohomology $H^*_c(\strict\Aut(\mathbb{G}_{1/n});\mathbb{F}_p)$ of the $p$-primary height $n$ strict Morava stabilizer group scheme is isomorphic, as a graded ring, to the cohomology of the solvable Lie $\mathbb{F}_p$-algebra $L(n,n)$.
\end{unnumberedthm:main2}
The most obvious corollary is that the mod $p$ cohomology of the height $n$ strict Morava stabilizer group scheme is algorithmically computable for $p>>n$.
We offer a few other corollaries as well. 
If the Smith-Toda complex $V(n-1)$ exists for all sufficiently large $p$, then 
Corollary \ref{smith-toda cor} gives us that the $K(n)$-local stable homotopy groups of $V(n-1)$ are isomorphic to the Lie algebra cohomology $H^*(L(n,n);\mathbb{F}_p)\otimes_{\mathbb{F}_p}\mathbb{F}_p[v_n^{\pm 1}]$, for $p>>n$.

Another corollary of our main theorems is Corollary \ref{new ss cor}: suppose that $n$ is a positive integer, and suppose that $p$ is a prime larger than the integer $N_n$ described in the statement of Theorem \ref{main thm 1}, as described above. Suppose that $X$ is an $E(n-1)$-acyclic finite CW-complex. Then there exist strongly convergent spectral sequences
\begin{align}
\label{new ss 1a} E_1^{*,*,*} \cong H^*(L(n,n);\mathbb{F}_p)\otimes_{\mathbb{F}_p} E_0E(n)_*(X)
  &\Rightarrow \Cotor_{E(n)_*E(n)}^{*,*}(E(n)_*,E(n)_*(X)),\\
\label{En-ass 1a} E_2^{*,*} \cong \Cotor_{E(n)_*E(n)}^{*,*}(E(n)_*,E(n)_*(X))
  &\Rightarrow \pi_*(L_{E(n)}X),
\end{align}
where $E_0E(n)_*(X)$ is the associated graded abelian group of the $p$-adic filtration 
\[ E(n)_*(X) \supseteq pE(n)_*(X)\supseteq p^2E(n)_*(X) \supseteq p^3E(n)_*(X)\supseteq \dots .
\]
Of course \eqref{En-ass 1a} is not new: it is simply the $E(n)$-Adams spectral sequence of $X$. The new part is spectral sequence \eqref{new ss 1a}, which together with \eqref{En-ass 1a}, gives us a means to pass from the cohomology of a solvable Lie algebra to the $E(n)$-local stable homotopy groups of a finite $E(n-1)$-acyclic CW-complex.

We think that the methods used to prove our main results are themselves of interest. The basic strategy is to find {\em integral deformations} of the objects involved in the construction of the Ravenel-May spectral sequence. Here is a sketch of the steps involved:
\begin{itemize}
\item For each positive integer $n$, in \cref{An integral lift of Ravenel...} we construct a Lie $\mathbb{Z}$-algebra $L_{int}(n)$, free over $\mathbb{Z}$, whose mod $p$ reduction for each $p$ is equal to the Lie $\mathbb{F}_p$-algebra $L(n,n)$. 
\item The $\mathbb{Z}$-linear dual of the universal enveloping $\mathbb{Z}$-algebra of $L_{int}(n)$ is a divided power $\mathbb{Z}$-algebra. Inside this divided power algebra, there is a natural polynomial Hopf $\mathbb{Z}$-algebra, which we call $\mathcal{Z}(n)$. 
\item We filter $\mathcal{Z}(n)$ in such a way that, for $p>>n$, upon reducing modulo $p$ and modulo all the $p$th powers of the polynomial generators, we recover Ravenel's filtration on $S(n,p)$. This filtration, and the definition of $\mathcal{Z}(n)$ itself, are in \cref{The Hopf algebra...}.
\item In \cref{The Hopf algebra...} we also show that the resulting spectral sequence 
\begin{align}\label{integral rmss 0}\Cotor_{\mathbb{F}_p\otimes_{\mathbb{Z}}E_0\mathcal{Z}(n)}^*(\mathbb{F}_p,\mathbb{F}_p) &\Rightarrow \Cotor_{\mathbb{F}_p\otimes_{\mathbb{Z}}\mathcal{Z}(n)}^*(\mathbb{F}_p,\mathbb{F}_p)\end{align} collapses immediately.
\item In \cref{RMSS collapse section}, we construct a set of cocycles in the cobar complex of $E_0\mathcal{Z}(n)$ which map, for $p>>n$, to an $\mathbb{F}_p$-linear basis for $\Cotor_{E_0S(n,p)}^*(\mathbb{F}_p,\mathbb{F}_p)\cong H^*(L(n,n);\mathbb{F}_p)$. 
\item We then show that these cocycles survive the integral Ravenel-May spectral sequence \eqref{integral rmss 0}, and that their images in the Ravenel-May spectral sequence \eqref{rmss 0} must also survive. Hence, for $p>>n$, all elements of $H^*(L(n,n);\mathbb{F}_p)$ must survive the Ravenel-May spectral sequence, i.e., the spectral sequence collapses immediately with no differentials.
\end{itemize}
One corollary is that the cohomology of our $n^2$-dimensional Lie $\mathbb{Z}$-algebra $L_{int}(n)$ has, upon reduction modulo $p$, the same cohomology as $H^*_c(\strict\Aut(\mathbb{G}_{1/n};\mathbb{F}_p)$ for $p>>n$. In that sense, if one makes the single calculation of the Lie algebra cohomology of $L_{int}(n)$, then one knows the mod $p$ cohomology of the height $n$ strict Morava stabilizer group scheme for all but finitely many $p$. We demonstrate with explicit integral calculations for $n=1$ and $n=2$ in \cref{examples section}.

We also define an integral deformation $S_{int}(n)^{\bullet}$ of the cobar complex of the Morava stabilizer algebra $S(n,p)$ itself. However, to prove the main results, the smaller object $\mathcal{Z}(n)$ (whose cobar complex is a sub-DGA of $S_{int}(n)^{\bullet}$) suffices. Since it is not used in the proofs of the main results, we consign the definition of $S_{int}(n)^{\bullet}$ to an appendix.

We remark that Theorems \ref{main thm 1} and \ref{main thm 2} are used in our joint work with Mohammad Behzad Kang, currently in preparation, in which we show that the cohomology of the height $n$ {\em full} Morava stabilizer group, with trivial mod $p$ coefficients, is isomorphic to the cohomology $H^*(U(n);\mathbb{F}_p)$ of the unitary group for all $p>>n$. 

The author thanks D. Ravenel for many valuable discussions about chapter 6 of \cite{MR860042} when the author was a student.

\subsection{Conventions.}

\begin{itemize}
\item The symbol $C_n$ denotes a cyclic group of order $n$. Throughout, we fix a choice of generator $\sigma$ for $C_n$.
\item The symbol $\floor{i}$ denotes the integer floor of a rational number $i$.
\item For each prime $p$ and each positive integer $n$, the height $n$ Morava stabilizer algebra is a certain Hopf $\mathbb{F}_p$-algebra. It is denoted $S(n)$ in the standard reference, \cite{MR860042}, but this suppresses the choice of prime $p$ from the notation. It will be convenient to include the choice of $p$ in the notation for the Morava stabilizer algebra, so we will write $S(n,p)$ rather than $S(n)$. 
\item Throughout, all formal group laws will be implicitly understood to be one-dimensional.
\item Given a prime number $p$ and a positive integer $n$, we use the symbol $\mathbb{G}_{1/n}$ for the $p$-typical height $n$ formal group law over $\mathbb{F}_p$ classified by the ring map $BP_*\rightarrow \mathbb{F}_p$ sending the Hazewinkel generator $v_n\in BP_*$ to $1$ and sending the other Hazewinkel generators $v_i$ to zero. (Sometimes this formal group law is called the ``height $n$ Honda formal group law.'') The symbol $\mathbb{G}_{1/n}$ comes from the Dieudonn\'{e}-Manin classification \cite{MR0157972} of $p$-divisible groups over a separably closed perfect field.
\item In this paper we will need to consider the cohomology of restricted Lie algebras and also of their underlying unrestricted Lie algebras over fields of positive characteristic. We write $H^*_{res}(\mathfrak{g},\rho)$ and $H^*(\mathfrak{g},\rho)$ for the restricted Lie algebra cohomology $\Ext_{V\mathfrak{g}}^*(k,\rho)$ and the unrestricted Lie algebra cohomology $\Ext_{U\mathfrak{g}}^*(k,\rho)$, respectively, with coefficients in a representation $\rho$ of $\mathfrak{g}$. 
\end{itemize}

\section{Review of Ravenel's Lie algebras and associated spectral sequences.}
\label{Review of Ravenel...}

This section is a review of known results, for readers unfamiliar with Ravenel's computational approach to the cohomology of the Morava stabilizer algebras. Readers already familiar with the ideas from chapter 6 of \cite{MR860042} can skip ahead to \cref{An integral lift of Ravenel...}.

The  $p$-primary height $n$ Morava stabilizer algebra $S(n,p)$ is isomorphic, as a $\mathbb{Z}/2(p^n-1)\mathbb{Z}$-graded $\mathbb{F}_p$-algebra, to 
\begin{equation}\label{snp pres 1} \mathbb{F}_p[t_1, t_2, \dots ]/(t_i^{p^n}-t_i\ \forall i),\end{equation} with $t_i$ in degree $2(p^i-1)$ \cite{MR0420619}. 
The algebra $S(n,p)$ is a $C_n$-equivariant commutative Hopf algebra over $\mathbb{F}_p$, where $C_n$ acts by the Frobenius map. The prime spectrum $\Spec S(n,p)$ represents a profinite group scheme, the strict automorphism group scheme of the $p$-height $n$ formal group law $\mathbb{G}_{1/n}$. This group scheme is pro-\'{e}tale, i.e., after a separable base change\footnote{In fact, base change to $\mathbb{F}_{p^n}$ suffices.}, the Hopf algebra $S(n,p)$ becomes isomorphic to the continuous linear dual of the group ring of a profinite group---an actual {\em group}, not just a group scheme. This profinite group is a pro-$p$-group, isomorphic to a maximal pro-$p$-subgroup of the group of units in the maximal order in an invariant $1/n$ central division algebra over $\mathbb{Q}_p$, and isomorphic also to the strict automorphism group of any height $n$ formal group law over $\overline{\mathbb{F}}_p$.

In section 1 of \cite{ravenel1977cohomology} (see chapter 6 of \cite{MR860042} for a textbook reference), Ravenel puts an increasing filtration on the Morava stabilizer algebra $S(n,p)$ as follows:
\begin{itemize}
\item Let $d_{n,i}$ be the integer defined recursively by the rule 
\begin{align*}
 d_{n,i} &= \left\{ \begin{array}{ll} 0 &\mbox{\ if\ } i\leq 0 \\ \max\{ i,pd_{n,i-n}\} &\mbox{\ if\ } i>0.\end{array}\right. \end{align*}
\item Let the filtration degree of $t_i^{p^j}$ be $d_{n,i}$ for each $j$. More generally, let the filtration degree of $t_i^k$ be $d_{n,i}$ times the sum of the coefficients in the base $p$ expansion of $k$. The filtration is otherwise multiplicative, e.g. the filtration degree of $t_it_j$ is equal to $d_{n,i}$ times $d_{n,j}$, if $i\neq j$.
\end{itemize}
The associated graded Hopf algebra $E_0S(n,p)$ is isomorphic, as a graded Hopf algebra, to $\mathbb{F}_p[t_{i,j}: i\geq 1, j\in\mathbb{Z}/n\mathbb{Z}]/t_{i,j}^p$, with $t_{i,j}$ in degree $2(p^i-1)p^j$. The element $t_{i,j}$ of $E_0S(n,p)$ represents the element $t_i^{p^j}$ of $S(n,p)$.

Ravenel's filtration has the agreeable property that its associated graded Hopf algebra $E_0S(n,p)$ is the linear dual of a primitively generated Hopf algebra, so the methods of Milnor--Moore \cite{MR0174052} and May's thesis \cite{MR2614527} can be applied. In particular, $E_0S(n,p)$ is dual to the restricted enveloping algebra $VP\left( E_0S(n,p)^*\right)$ of the restricted Lie $\mathbb{F}_p$-algebra $P\left( E_0S(n,p)^*\right)$ of primitives in the dual Hopf algebra $E_0S(n,p)^*$. In \cite{MR0193126},\cite{MR2614527}, and \cite{MR0185595}, May constructed a spectral sequence
\begin{align}
\label{lmss}  E_1^{*,*,*} \cong H^*(\mathfrak{g},k) \otimes_k P(\mathfrak{g}^*) &\Rightarrow H^*_{res}(\mathfrak{g},k)\\
\nonumber d_r: E_r^{s,t,u} &\rightarrow E_r^{s+1,t+r,u}
\end{align}
for every graded restricted Lie algebra $\mathfrak{g}$, concentrated in even degrees, over a field $k$ of characteristic $p$. Here $P(\mathfrak{g}^*)$ is the free commutative $k$-algebra on the dual $k$-vector space of $\mathfrak{g}$. 
The grading is as follows: elements of $H^s(\mathfrak{g},k)$ in internal\footnote{We follow the standard convention of referring to the grading on $H^*(\mathfrak{g},k)$ induced by the grading on $\mathfrak{g}$ itself as the {\em internal} grading.} grading $u$ are in tridegree $(s,0,u)$. Nonzero elements of $\mathfrak{g}^*\subseteq P(\mathfrak{g}^*)$ of internal degree $u$ are in tridegree $(2,1,pu)$.
Elements in tridegree $(s,t,u)$ contribute, in the abutment, to elements of $H^s_{res}(\mathfrak{g},k)$ of internal degree $u$.

Spectral sequence \eqref{lmss} was one of two new spectral sequences studied in May's thesis. In the literature, both spectral sequences have at times been called ``the May spectral sequence.'' To resolve the ambiguity in naming, we will refer to \eqref{lmss} as the {\em Lie-May spectral sequence}, since it is the spectral sequence from May's thesis which relates restricted Lie algebra cohomology to unrestricted Lie algebra cohomology. (The other spectral sequence studied in May's thesis was the spectral sequence of a Hopf algebra, and particularly the $2$-primary Steenrod algebra, filtered by powers of its augmentation ideal.)

In Theorem 1.6 of \cite{ravenel1977cohomology} (cf. \cite[Theorem 6.3.5]{MR860042}), Ravenel shows that spectral sequence \eqref{lmss} admits a tensor splitting in the case that $\mathfrak{g} = P\left( E_0S(n,p)^*\right)$. One tensor factor converges to $\mathbb{F}_p$ concentrated in tridegree $(0,0,0)$, while the other tensor factor has $E_1$-term isomorphic to 
\begin{align*}
 H^*\left(L\left( n,\floor{\frac{pn}{p-1}}\right);\mathbb{F}_p\right) \otimes_{\mathbb{F}_p} P\left(b_{i,j}: 1\leq i\leq \floor{\frac{n}{p-1}}, j\in \mathbb{Z}/n\mathbb{Z}\right),
\end{align*}
where the notation and gradings are as follows:
\begin{itemize}
\item 
$L\left(n,\floor{\frac{pn}{p-1}}\right)$ is a certain $n\floor{\frac{pn}{p-1}}$-dimensional quotient Lie algebra of $P\left( E_0S(n,p)^*\right)$,
\item elements of $H^s\left(L\left( n,\floor{\frac{pn}{p-1}}\right);\mathbb{F}_p\right)$ in internal grading $u$ are in tridegree $(s,0,u)$,
\item and the polynomial generator $b_{i,j}$ is in tridegree $(2,1,2p^{i+1}(p^j-1))$. 
\end{itemize}
Consequently Ravenel gets a spectral sequence
\begin{align}
\label{lmss 2}  E_1^{*,*,*} &\cong H^*\left(L\left( n,\floor{\frac{pn}{p-1}}\right);\mathbb{F}_p\right) 
 \otimes_{\mathbb{F}_p} P\left(b_{i,j}: 1\leq i\leq \floor{\frac{n}{p-1}}, j\in \mathbb{Z}/n\mathbb{Z}\right) \\
\nonumber &\Rightarrow H^*_{res}\left(P\left( E_0S(n,p)^*\right); \mathbb{F}_p\right) \\
\nonumber &\cong \Cotor^*_{E_0S(n,p)}(\mathbb{F}_p,\mathbb{F}_p) \\
\nonumber d_r: E_r^{s,t,u} &\rightarrow E_r^{s+1,t+r,u}.
\end{align}
We will refer to \eqref{lmss 2} as the {\em Ravenel-Lie-May spectral sequence,} or RLMSS for short.

It will be useful to have an explicit description of the Lie $\mathbb{F}_p$-algebra $L(n,m)$ for positive integers $n,m$. 
In Theorem 1.4 of \cite{ravenel1977cohomology} (cf. \cite[Theorem 6.3.3]{MR860042}), Ravenel shows that the restricted Lie $\mathbb{F}_p$-algebra $P(E_0S(n,p)^*)$ has $\mathbb{F}_p$-linear basis
$\left\{ x_{i,j} : i\geq 1, j\in \mathbb{Z}/n\mathbb{Z}\right\},$
with Lie bracket 
\begin{align} 
\label{lie bracket 1} [x_{i,j},x_{k,\ell}] &= \left\{ \begin{array}{ll} \delta_{i+j}^{\ell} x_{i+k,j} - \delta_{k+\ell}^j x_{i+k,\ell} &\mbox{\ if\ } i+k\leq \floor{\frac{pn}{p-1}} \\ 0 &\mbox{\ otherwise} \end{array}\right. .
\end{align}
The Lie $\mathbb{F}_p$-algebra $L(n,m)$ is defined to be the quotient of $P(E_0S(n,p)^*)$ by the linear span of the elements $x_{i,j}$ satisfying $i > m$. 

We also have the spectral sequence 
\begin{align}
\label{rmss}  E_1^{*,*,*} &\cong H^*_{res}\left(P\left( E_0S(n,p)^*\right); \mathbb{F}_p\right) \\
\nonumber &\cong \Cotor^*_{E_0S(n,p)}(\mathbb{F}_p,\mathbb{F}_p) \\ 
\nonumber &\Rightarrow \Cotor^*_{S(n,p)}(\mathbb{F}_p,\mathbb{F}_p) \\
\label{rmss diff} d_r: E_r^{s,t,u} &\rightarrow E_r^{s+1,t,u-r}
\end{align}
arising from applying Ravenel's filtration to the cobar complex of $S(n,p)$. Here $s$ is the cohomological degree, $t$ the internal degree, and $u$ the Ravenel filtration degree. 
As in \cref{Introduction.}, we will call \eqref{rmss} the {\em Ravenel-May spectral sequence,} or RMSS for short. 

If $p>n+1$, then $\floor{\frac{pn}{p-1}} = n$, and consequently the polynomial factor in the $E_1$-term of the RLMSS is trivial. Hence the RLMSS collapses, yielding an isomorphism of graded rings $H^*\left(L(n,n);\mathbb{F}_p\right) \cong H^*_{res}\left(P\left( E_0S(n,p)^*\right); \mathbb{F}_p\right)$. This is why, for $p>n+1$, the $E_1$-term of the RMSS is simply the cohomology of the Lie algebra $L(n,n)$, as described above in \eqref{rmss 0}.

The RMSS is known to have nonzero differentials for some values of $p$ and $n$ satisfying $p\leq n+1$. For example, in the case $p=2, n=3$, we have the nonzero Ravenel-May differential $d_1(b_{20}) = h_{11}b_{11} + h_{12}b_{10}$, as in\footnote{To be clear, in \cite{MR860042}, the indexing on the RMSS is slightly different from the indexing given above in \eqref{rmss diff}. In the indexing used in \cite{MR860042}, the differential on $b_{20}$ is a $d_2$ rather than a $d_1$.} Theorem 6.3.14 of \cite{MR860042}. However, it is not known whether there exist any nonzero RMSS differentials for $p>n+1$.

Here is some explanation of how Ravenel-May differentials arise. The coproduct in $S(n,p)$ is given by 
\begin{align} \label{coprod formula 1} \Delta(t_i) &= \sum_{k=0}^i t_k\otimes t_{i-k}^{p^k}\end{align} 
for $i\leq n$, while the coproduct in $E_0S(n,p)$ is given by 
\begin{align} \label{coprod formula 2} \Delta(t_{i,j}) &= \sum_{k=0}^i t_{k,j}\otimes t_{i-k,j+k}\end{align} 
for $i\leq n$. The element $t_{i,j}$ in $E_0S(n,p)$ represents the element $t_i^{p^j}$ in $S(n,p)$, so \eqref{coprod formula 2} is a simple consequence of \eqref{coprod formula 1}. 
If $i>n$, then the coproduct on $t_i$ in $S(n,p)$, and on $t_{i,j}$ in $E_0S(n,p)$, is more complicated than the formulas \eqref{coprod formula 1} and \eqref{coprod formula 2}. However, if $p>n+1$, then it is the coproduct on $t_i$ for $i\leq n$ which is responsible for the $\Cotor$-groups, in the sense that the the cohomology of $L(n,n)$ agrees with $\Cotor$ over $E_0S(n,p)$, and the Lie bracket \eqref{lie bracket 1} in $L(n,n)$ is dual to the coproduct on elements $t_{i,j}$ in $E_0S(n,p)$ with $i\leq n$.

To summarize: the Hopf algebra $E_0S(n,p)$ can be generated by elements $\{ t_{i,j}\}$ which are representable by generators $\{ t_i^{p^j}\}$ for $S(n,p)$, and on the generators whose coproducts determine the $\Cotor$ groups, the coproduct in $E_0S(n,p)$ agrees with the coproduct in $S(n,p)$. At a glance, this makes it seem like $S(n,p)$ and $E_0S(n,p)$ must have the same $\Cotor$-groups. In fact, this was Ravenel's argument from \cite{ravenel1977cohomology}, which he later \cite{MR860042} pointed out was not correct. The trouble is the difference between the relations in $E_0S(n,p)$ and in $S(n,p)$. In $S(n,p)$, the $p^n$th power of each generator $t_i$ is equal to $t_i$ itself. Meanwhile, in $E_0S(n,p)$, the $p$th power of each generator $t_{i,j}$ is zero. 
As a consequence, a cocycle in the cobar complex for $E_0S(n,p)$ can involve linear combinations of tensor products of polynomials in the generators $t_{i,j}$ with the property that, upon applying the cobar complex differential $d$, each term involves a $p$th power of some generator, hence is zero. If we lift such an element to the cobar complex of $S(n,p)$ by replacing each instance of $t_{i,j}$ with $t_i^{p^j}$, it is an {\em a priori} possibility that applying $d$ to the resulting cochain does not yield zero, since $p$th powers of generators are no longer zero: for example in $S(1,p)$ we have $t_i^p = t_i$ instead of $t_i^p = 0$. 

This phenomenon---a cocycle in the cobar complex of $S(n,p)$ lifting to a non-cocycle in the cobar complex of $E_0S(n,p)$ due to the difference between the multiplicative relations in $S(n,p)$ and in $E_0S(n,p)$---would occur whenever there is a nonzero differential in the RMSS for $p>n+1$, if such differentials indeed exist.

\section{A simultaneous integral lift of Ravenel's height $n$ Lie algebras for all primes.}
\label{An integral lift of Ravenel...}

Fix a positive integer $n$. For each prime $p$, we have the Lie $\mathbb{F}_p$-algebra $L(n,n)$. There exists a simultaneous integral lift of all these Lie algebras, for fixed $n$:
\begin{definition}
Let $L_{int}(n,n)$ denote the $C_n$-equivariant Lie $\mathbb{Z}$-algebra with $\mathbb{Z}$-linear basis 
\[\left\{ x_{i,j}: i\in \{ 1, \dots ,n\}, j\in\mathbb{Z}/n\mathbb{Z}\right\}\]
where $C_n$ acts freely by letting $\sigma x_{i,j} = x_{i,j+1}$, and where the Lie bracket is defined by the rule \eqref{lie bracket 1}, above.
\end{definition}
While $L_{int}(n,n)$ is $n^2$-dimensional as a Lie $\mathbb{Z}$-algebra, it is only $n$-dimensional as a Lie $\mathbb{Z}[C_n]$-algebra. To be clear, Lie algebra cohomology {\em does} depend on the choice of ground ring\footnote{Consider how this goes in the more familiar situation over $\mathbb{F}_p$, rather than over $\mathbb{Z}$. The Lie algebra $L(2,2)$ is four-dimensional and unimodular, hence $H^4_{unr}(L(2,2);\mathbb{F}_p)$ is nontrivial, spanned by $h_{1,0}h_{1,1}h_{2,0}h_{2,1}$. However, if we consider $L(2,2)$ as a Lie $\mathbb{F}_p[C_2]$-algebra rather than as a Lie $\mathbb{F}_p$-algebra, then it is only two-dimensional, hence its Chevalley-Eilenberg complex (over $\mathbb{F}_p[C_2]$!) is trivial above cohomological dimension $2$, and certainly does not have vanishing $H^4$. It is the cohomology over $\mathbb{F}_p$, not over $\mathbb{F}_p[C_n]$, which is always considered in applications to stable homotopy.}, and throughout, whenever we speak of the cohomology of $L_{int}(n,n)$, we shall always mean its cohomology as a Lie $\mathbb{Z}$-algebra, not as a Lie $\mathbb{Z}[C_n]$-algebra. 

The Lie $\mathbb{Z}$-algebra $L_{int}(n,n)$ is free and finite-dimensional over $\mathbb{Z}$, and its reduction modulo a prime $p$ is Ravenel's Lie $\mathbb{F}_p$-algebra $L(n,n)$. We recall the relationship in cohomology between  a finite-dimensional $\mathbb{Z}$-free Lie $\mathbb{Z}$-algebra and its mod $p$ reduction. The author has never seen this relationship in print, nor heard it mentioned, so we offer a proof, but the proof is easy, and the result must surely be well-known. This relationship is used later, in the proof of Lemma \ref{ss map is surj}.
\begin{prop}\label{univ coeff seq}
Let $L$ be a Lie $\mathbb{Z}$-algebra which is finite-dimensional and free over $\mathbb{Z}$. Then, for each prime $p$ and each integer $m$, the natural map $H^m(L;\mathbb{Z})\otimes_{\mathbb{Z}}\mathbb{F}_p \rightarrow H^m(L\otimes_{\mathbb{Z}}\mathbb{F}_p; \mathbb{F}_p)$ fits into a short exact sequence
\begin{align}
\label{ses 203} 0 \rightarrow H^m(L;\mathbb{Z})\otimes_{\mathbb{Z}}\mathbb{F}_p \rightarrow H^m(L\otimes_{\mathbb{Z}}\mathbb{F}_p; \mathbb{F}_p) \rightarrow \Tor_1^{\mathbb{Z}}(H^{m+1}(L;\mathbb{Z}),\mathbb{F}_p) \rightarrow 0.
\end{align}
\end{prop}
\begin{proof}
Recall that the Chevalley-Eilenberg complex of $L$ is the $\mathbb{Z}$-linear dual of the exterior $\mathbb{Z}$-algebra on $UL$. Exterior powers and the universal enveloping algebra functor $U$ each commute with base-change, so we have isomorphisms
\begin{align}
\label{iso 1230} \hom_{\mathbb{Z}}(\Lambda^{\bullet}_{\mathbb{Z}}(L),\mathbb{Z})\otimes_{\mathbb{Z}} \mathbb{F}_p 
  &\cong \hom_{\mathbb{Z}}(\Lambda^{\bullet}_{\mathbb{Z}}(L),\mathbb{F}_p) \\
\nonumber  &\cong \hom_{\mathbb{F}_p}(\Lambda^{\bullet}_{\mathbb{Z}}(L)\otimes_{\mathbb{Z}}\mathbb{F}_p,\mathbb{F}_p) \\
\label{iso 1231}  &\cong \hom_{\mathbb{F}_p}\left(\Lambda^{\bullet}_{\mathbb{F}_p}(L\otimes_{\mathbb{Z}}\mathbb{F}_p),\mathbb{F}_p\right) .
\end{align}
The cohomology of cochain complex \eqref{iso 1231} is $H^*(L\otimes_{\mathbb{Z}}\mathbb{F}_p;\mathbb{F}_p)$, while the cohomology of cochain complex \eqref{iso 1230} is describable by the universal coefficient sequence relating the cohomology of a base-changed cochain complex to the base-changed cohomology of the cochain complex. The resulting exact sequence is \eqref{ses 203}.
\end{proof}

\section{The Hopf algebra $\mathcal{Z}(n)$.}
\label{The Hopf algebra...}

To introduce the Hopf algebra $\mathcal{Z}(n)$, it is convenient to begin with some general, well-known observations about the relationship between Lie algebras, divided power algebras, and polynomial algebras. Suppose that $L$ is a Lie $\mathbb{Z}$-algebra which is free and finite-dimensional over $\mathbb{Z}$. For any free $\mathbb{Z}$-module $A$, we will write $A^*$ for its $\mathbb{Z}$-linear dual $\hom_{\mathbb{Z}}(A,\mathbb{Z})$.
The dual $(UL)^*$ of the universal enveloping algebra $UL$ is, as a ring\footnote{The {\em coproduct} on $(UL)^*$ depends on the Lie bracket of $L$, but the {\em product} on $(UL)^*$ is totally insensitive to the bracket.}, the divided power $\mathbb{Z}$-algebra on any $\mathbb{Z}$-linear basis for $L^*$. We write $\Gamma_{\mathbb{Z}}(L^*)$ for this divided power $\mathbb{Z}$-algebra. Inside the divided power algebra, we have the subring generated by $L^*$, i.e., the free commutative $\mathbb{Z}$-algebra $\mathbb{Z}[L^*]$ on the $\mathbb{Z}$-module $L^*$. Put another way, $\mathbb{Z}[L^*]$ is the polynomial ring with one generator for each element in a fixed $\mathbb{Z}$-linear basis for $L^*$. The inclusion of Hopf $\mathbb{Z}$-algebras $\mathbb{Z}[L^*]\hookrightarrow\Gamma_{\mathbb{Z}}(L^*)$ induces a map 
\begin{align}
\label{cotor map} \Cotor^*_{\mathbb{Z}[L^*]}(\mathbb{Z},\mathbb{Z}) &\hookrightarrow \Cotor^*_{\Gamma_{\mathbb{Z}}(L^*)}(\mathbb{Z},\mathbb{Z})
\end{align}
which is generally not an isomorphism, although it becomes an isomorphism after rationalization. It is a nice exercise to calculate that the domain of \eqref{cotor map} has a lot of torsion elements which map to zero in the codomain of \eqref{cotor map}.

Now we consider the special case in which $n$ is a positive integer and $L$ is the Lie $\mathbb{Z}$-algebra $L_{int}(n,n)$. We have an isomorphism of rings 
\begin{align}
\label{hopf alg iso 3409}
 (UL_{int}(n,n))^* &\cong \Gamma_{\mathbb{Z}}\left(t_{i,j}: i\in \{ 1, \dots ,n\}, j\in\mathbb{Z}/n\mathbb{Z}\right),
\end{align}
where $t_{i,j}$ denotes the dual element of $x_{i,j}$. 
It is a routine calculation (essentially the same as that of Theorem 6.3.2 of \cite{MR860042}) to show that, under the isomorphism \eqref{hopf alg iso 3409}, the coproduct on the Hopf $\mathbb{Z}$-algebra $(UL_{int}(n,n))^*$ yields the coproduct on $\Gamma_{\mathbb{Z}}\left(t_{i,j}: i\in \{ 1, \dots ,n\}, j\in\mathbb{Z}/n\mathbb{Z}\right)$ given by the formula
$\Delta(t_{i,j}) = \sum_{k=0}^i t_{k,j} \otimes t_{i-k,j+k}$.

We have explained how, within the divided power algebra $(UL)^*$, there is the natural polynomial algebra $\mathbb{Z}[L^*]$. We now define the Hopf algebra $\mathcal{Z}(n)$ to be precisely that polynomial subalgebra of $(UL_{int}(n,n))^*$.
\begin{definition}
Let $\mathcal{Z}(n)$ be the $C_n$-equivariant commutative Hopf $\mathbb{Z}$-algebra which, as a commutative ring, is the subring of $\Gamma_{\mathbb{Z}}\left(t_{i,j}: i\in \{ 1, \dots ,n\}, j\in\mathbb{Z}/n\mathbb{Z}\right)$ generated by the elements $t_{i,j}$. 

Consequently, $\mathcal{Z}(n)$ has the following presentation:
\begin{itemize}
\item As a commutative $\mathbb{Z}[C_n]$-algebra, $\mathcal{Z}(n)$ is free on generators $\lowert_1, \lowert_2,\dots ,\lowert_n$.
\item The coproduct on $\mathcal{Z}(n)$ is given by
$\Delta(\lowert_i) = \sum_{j=0}^{i} \lowert_j \otimes\ \sigma^jt_{i-j}$.
\item The augmentation on $\mathcal{Z}(n)$ is given by $\epsilon(\lowert_i) = 0$ for all $i$.
\item The map $\mathcal{Z}(n)\hookrightarrow \Gamma_{\mathbb{Z}}\left(t_{i,j}: i\in \{ 1, \dots ,n\}, j\in\mathbb{Z}/n\mathbb{Z}\right)$ sends $\sigma^jt_i$ to $t_{i,j}$.
\end{itemize}
\end{definition}

\begin{definition}\label{def of rav filt on Zn}
Given integers $n\geq 1$ and $q\geq 2$, the {\em $q$-Ravenel filtration on $\mathcal{Z}(n)$} is the increasing filtration in which, for all $j$, the element $\sigma^jt_i$ is in degree $d_{n,q,i}$.
\end{definition}

For a prime $p$, write $E_0^p\mathcal{Z}(n)$ for the associated graded Hopf $\mathbb{Z}$-algebra of the $p$-Ravenel filtration on $\mathcal{Z}(n)$.
\begin{observation}\label{q-rav triviality}
For all $n$ and $p$, the product and coproduct on $\mathcal{Z}(n)$ each strictly preserve the $p$-Ravenel filtration. Consequently $E_0^p\mathcal{Z}(n) = \mathcal{Z}(n)$.
\end{observation}
We have a map of $C_n$-equivariant Hopf algebras $\mathcal{Z}(n)\otimes_{\mathbb{Z}}\mathbb{F}_p \rightarrow S(n,p)$ given by sending $\sigma^jt_i$ to $t_i^{p^j}$. This map sends the $p$-Ravenel filtration on $\mathcal{Z}(n)$ to the Ravenel filtration on $S(n,p)$. Each of these two filtrations yields a May-type spectral sequence. Hence we have a map of spectral sequences
\begin{equation}\label{ss map 1}\xymatrix{
 \Cotor^{*,*,*}_{E_0^p\left( \mathcal{Z}(n)\otimes_{\mathbb{Z}}\mathbb{F}_p\right)}(\mathbb{F}_p,\mathbb{F}_p) \ar@{=>}[r]\ar[d] &
  \Cotor^{*,*}_{\mathcal{Z}(n)\otimes_{\mathbb{Z}}\mathbb{F}_p}(\mathbb{F}_p,\mathbb{F}_p) \ar[d] \\
 \Cotor^{*,*,*}_{E_0S(n,p)}(\mathbb{F}_p,\mathbb{F}_p) \ar@{=>}[r] &
  \Cotor^{*,*}_{S(n,p)}(\mathbb{F}_p,\mathbb{F}_p) .
}\end{equation}

\section{The Ravenel-May spectral sequence collapses immediately at large primes.}
\label{RMSS collapse section}

Finally, in this section we use the definitions in the preceding sections to actually prove something new. 
\begin{lemma}\label{ss map is surj}
Fix a positive integer $n$. Then there exists an integer $N$ such that, if $p>N$, then the left-hand vertical map in diagram \eqref{ss map 1} is surjective. 
\end{lemma}
\begin{proof}
Choose a $\mathbb{Q}$-linear basis $\mathcal{B}$ for the rational Lie algebra cohomology $H^*(\mathbb{Q}\otimes_{\mathbb{Z}} L_{int}(n,n);\mathbb{Q})$. For each element $b\in \mathcal{B}$, choose a cocycle representative $\tilde{b}$ for $b$ in the cobar complex of the $\mathbb{Q}$-linear dual $U(\mathbb{Q}\otimes_{\mathbb{Z}}L_{int}(n,n))^*$ of the universal enveloping $\mathbb{Q}$-algebra of $\mathbb{Q}\otimes_{\mathbb{Z}}L_{int}(n,n)$. By exactness of rationalization, that cobar complex is the rationalization of the cobar complex of the $\mathbb{Z}$-linear dual $(UL_{int}(n,n))^*$ of the universal enveloping $\mathbb{Z}$-algebra of $L_{int}(n,n)$. 

Since $\mathcal{B}$ is finite, we may clear denominators to pass from the rationalization of the cobar complex of $(UL_{int}(n,n))^*$ to the cobar complex of $(UL_{int}(n,n))^*$ itself. Write $D$ for the least common multiple of the denominators of the chosen elements $\{ \tilde{b}: b\in\mathcal{B}\}$ of the cobar complex of $(\mathbb{Q}\otimes_{\mathbb{Z}}UL_{int}(n,n))^* \cong \mathbb{Q}\otimes_{\mathbb{Z}} \Gamma_{\mathbb{Z}}(L_{int}(n,n)^*)$. Write $D\cdot \tilde{\mathcal{B}}$ for the subset $\{ D\cdot \tilde{b} : b\in \mathcal{B}\}$ of $\Gamma_{\mathbb{Z}}(L_{int}(n,n)^*)$. 

Finally, for each element $D\cdot \tilde{b}$, there exist only finitely many denominators in the expression of $D\cdot \tilde{b}$ in terms of divided powers of the elements of $L_{int}(n,n)^*$. Hence, again using the finiteness of $\mathcal{B}$, we may also clear the denominators in all those divided powers. That is, there exists a positive integer $E$ such that 
\begin{align*}
 E\cdot D\cdot \tilde{\mathcal{B}}
  &:= \{ E\cdot D\cdot \tilde{b} : b\in \mathcal{B}\} \\
  &\subseteq \mathbb{Z}[L_{int}(n,n)^*] \\
  &\subseteq \Gamma_{\mathbb{Z}}(L_{int}(n,n)^*).
\end{align*}

For any prime $p>n+1$, Ravenel proved (see \cref{Review of Ravenel...}, above) that the map of Hopf $\mathbb{F}_p$-algebras 
\begin{align*}
 (UL(n,n))^* &\rightarrow 
 E_0S(n,p)\\
 t_{i,j} &\mapsto t_{i,j}
\end{align*}
induces an isomorphism 
\begin{align}
\label{iso 23240}\Cotor_{(UL(n,n))^*}^*(\mathbb{F}_p,\mathbb{F}_p) &\rightarrow \Cotor_{E_0S(n,p)}^*(\mathbb{F}_p,\mathbb{F}_p).\end{align}
The domain of \eqref{iso 23240} is the Lie algebra cohomology $H^*(L(n,n);\mathbb{F}_p)$. We now make use of Proposition \ref{univ coeff seq}: since $L_{int}(n,n)$ is a finite-dimensional Lie $\mathbb{Z}$-algebra, its cohomology $H^*(L_{int}(n,n);\mathbb{Z})$ is a finitely generated abelian group, hence has $\ell$-torsion for only finitely many primes $\ell$. Hence, if we choose $p$ sufficiently large, then the $\Tor_1$-group in Proposition \ref{univ coeff seq} vanishes, yielding an isomorphism 
\begin{align*}
 H^*(L_{int}(n,n);\mathbb{Z})\otimes_{\mathbb{Z}}\mathbb{F}_p &\stackrel{\cong}{\longrightarrow} H^*(L(n,n);\mathbb{F}_p).\end{align*} 

The set of cocycles $E\cdot D\cdot \tilde{\mathcal{B}}$ in the cobar complex of $\mathcal{Z}(n) = \mathbb{Z}[L_{int}(n,n)^*]$ is finite, and each member $E\cdot D\cdot \tilde{b}$ of $E\cdot D\cdot \tilde{\mathcal{B}}$ is a sum of rank $1$ tensors, each of which is a tensor of elements in $\mathbb{Z}[L_{int}(n,n)^*]$, each of which has only finitely many coefficients in $\mathbb{Z}$. The point is that there can be only finitely many primes that divide any of these coefficients in any of these rank $1$ tensors in any element $E\cdot D\cdot \tilde{b}$ of $E\cdot D \cdot \tilde{\mathcal{B}}$. Consequently, if $p$ is sufficiently large, then the only $\mathbb{Z}$-linear combinations $\alpha_1 E\cdot D\cdot \tilde{b}_1 + \dots + \alpha_m E\cdot D\cdot \tilde{b}_m$ of the members of $E\cdot D\cdot \tilde{\mathbb{B}}$ which map to zero in $H^*(L_{int}(n,n)\otimes_{\mathbb{Z}}\mathbb{F}_p;\mathbb{F}_p)$ are those in which all the coefficients $\alpha_1, \dots ,\alpha_m$ are divisible by $p$. 

Consequently, for $p>>n$, the set of cocycles $E\cdot D \cdot \tilde{\mathcal{B}}$ in the cobar complex of $\mathcal{Z}(n)$ maps to a $\mathbb{F}_p$-linear basis for $\Cotor_{E_0S(n,p)}^*(\mathbb{F}_p,\mathbb{F}_p)$. Consequently the classes in $\Cotor_{\mathcal{Z}(n)\otimes_{\mathbb{Z}}\mathbb{F}_p}^*(\mathbb{F}_p,\mathbb{F}_p)$ represented by the members of $E\cdot D\cdot \tilde{\mathcal{B}}$ surject on to $\Cotor_{E_0S(n,p)}^*(\mathbb{F}_p,\mathbb{F}_p)$. Using Observation \ref{q-rav triviality} to identify $\mathcal{Z}(n)\otimes_{\mathbb{Z}}\mathbb{F}_p$ and $E_0^p(\mathcal{Z}(n)\otimes_{\mathbb{Z}}\mathbb{F}_p) \cong (E_0^p\mathcal{Z}(n))\otimes_{\mathbb{Z}}\mathbb{F}_p$, the left-hand vertical map in \eqref{ss map 1} is surjective.
\end{proof}

\begin{theorem}\label{main thm 1}
Fix a positive integer $n$. Then there exists some integer $N$ such that, for all $p>N$, the Ravenel-May spectral sequence 
\begin{align}
\nonumber  E_1^{*,*,*} 
 \cong \Cotor^*_{E_0S(n,p)}(\mathbb{F}_p,\mathbb{F}_p) 
 &\Rightarrow \Cotor^*_{S(n,p)}(\mathbb{F}_p,\mathbb{F}_p) 
\end{align}
has no nonzero differentials.
\end{theorem}
\begin{proof}
Using Lemma \ref{ss map is surj}, choose $p$ large enough that the left-hand vertical map in \eqref{ss map 1} is surjective. The spectral sequence drawn along the top horizontal edge of diagram \eqref{ss map 1} collapses immediately with no differentials, by Observation \ref{q-rav triviality}. It is elementary to see that, if a map of spectral sequence is surjective on some page, and the domain spectral sequence has no nonzero differentials on and after that page, then the codomain spectral sequence also has no nonzero differentials on and after that page.
Hence the spectral sequence drawn along the bottom horizontal edge of \eqref{ss map 1}---i.e., the Ravenel-May spectral sequence---collapses with no nonzero differentials at the $E_1$-page.
\end{proof}

\begin{theorem}\label{main thm 2}
Fix a positive integer $n$. Then there exists some integer $N$ such that, for all $p>N$, the mod $p$ continuous cohomology $H^*_c(\strict\Aut(\mathbb{G}_{1/n});\mathbb{F}_p)$ of the height $n$ strict Morava stabilizer group scheme is isomorphic, as a graded ring, to the cohomology $H^*(L(n,n);\mathbb{F}_p)$ of the solvable Lie $\mathbb{F}_p$-algebra $L(n,n)$.
\end{theorem}
\begin{proof}
For $p>n+1$, we have the RMSS $H^*(L(n,n);\mathbb{F}_p)\Rightarrow H^*_c(\strict\Aut(\mathbb{G}_{1/n});\mathbb{F}_p)$. By Theorem \ref{main thm 1}, the spectral sequence collapses with no differentials for $p>>n$, so for sufficiently large $p$, $H^*(L(n,n);\mathbb{F}_p)$ and $H^*_c(\strict\Aut(\mathbb{G}_{1/n});\mathbb{F}_p)$ are isomorphic as graded abelian groups.

We need to show that $H^*(L(n,n);\mathbb{F}_p)$ and $H^*_c(\strict\Aut(\mathbb{G}_{1/n});\mathbb{F}_p)$ are also isomorphic as graded {\em rings.} The RMSS is a spectral sequence of algebras, so we only need to verify that there are no multiplicative filtration jumps in the abutment of the spectral sequence. The argument here is much the same as used in the proof of Theorem \ref{main thm 1}: we have the map of spectral sequences \eqref{ss map 1}, which is surjective on the $E_1$-page for $p>>n$, by Lemma \ref{ss map is surj}. The domain of the map of spectral sequences collapses immediately with no nonzero differentials, and furthermore has no multiplicative filtration jumps, since $\mathcal{Z}(n) = E_0^p\mathcal{Z}(n)$ by Observation \ref{q-rav triviality}. Hence for $p>>n$, there are no multiplicative filtration jumps.
\end{proof}

\begin{corollary}\label{smith-toda cor}
Let $n$ be a positive integer. 
Suppose that the Smith-Toda complex $V(n-1)$ exists for all $p>>n$. Equip $H^*(L(n,n);\mathbb{F}_p)$ with the Adams grading, i.e., the grading in which an element in cohomological degree $d$ and internal degree $i$ is in Adams degree $i-d$.
Then, for $p>>n$, the $K(n)$-local homotopy groups $\pi_*(L_{K(n)}V(n-1))$ of $V(n-1)$ are isomorphic as a graded abelian group to $H^*(L(n,n);\mathbb{F}_p)\otimes_{\mathbb{F}_p}\mathbb{F}_p[v_n^{\pm 1}]$, with $v_n$ in degree $2(p^n-1)$ and with the Adams degree as the grading on $H^*(L(n,n);\mathbb{F}_p)$.

If it is furthermore true that $V(n-1)$ is a homotopy-associative ring spectrum for $p>>n$, then the isomorphism $\pi_*(L_{K(n)}V(n-1)) \cong H^*(L(n,n);\mathbb{F}_p)\otimes_{\mathbb{F}_p}\mathbb{F}_p[v_n^{\pm 1}]$ for $p>>n$ is an isomorphism of graded rings.
\end{corollary}
\begin{proof}
Whenever $V(n-1)$ exists, we have a spectral sequence 
\begin{align}
\label{ass 1} E_2^{s,t} \cong \left(\Cotor_{S(n,p)}^{s,*}(\mathbb{F}_p,\mathbb{F}_p)\otimes_{\mathbb{F}_p} K(n)_*\right)^t 
  &\Rightarrow \pi_{t-s}(L_{K(n)}V(n-1))\\
\nonumber d_r: E_r^{s,t} &\rightarrow E_r^{s+r,t+r-1}.
\end{align}
Spectral sequence \eqref{ass 1} is most simply constructed as the $E(n)$-Adams spectral sequence of $V(n-1)$, which converges to $\pi_*(L_{E(n)}V(n-1)) \cong \pi_*(L_{K(n)}V(n-1))$. The Morava-Miller-Ravenel change-of-rings isomorphism (see \cite{MR0458410} or \cite{MR782555} or section 6.1 of \cite{MR860042}) identifies the $E_2$-page $\Cotor_{E(n)_*E(n)}^{*,*}(E(n)_*,E(n)_*V(n-1))$ of the $E(n)$-Adams spectral sequence for $V(n-1)$ with the continuous cohomology $H^*_c(\strict\Aut(\mathbb{G}_{1/n});K(n)_*) \cong H^*_c(\strict\Aut(\mathbb{G}_{1/n});\mathbb{F}_p)\otimes_{\mathbb{F}_p} K(n)_*$.

The $s$-line in the $E_2$-term of spectral sequence \eqref{ass 1} is concentrated in Adams degrees congruent to $-s$ modulo $2p-2$. If $p>\frac{n^2+n+2}{2}$, then an elementary bidegree argument shows that any $E(n)$-Adams differential would have to cross the horizontal vanishing line at $s=n^2+n$, hence must be zero. The same bidegree argument also rules out multiplicative filtration jumps in the $E(n)$-Adams spectral sequence when $p>\frac{n^2+n+2}{2}$. (This argument for collapse of spectral sequence \eqref{ass 1} for $p>\frac{n^2+n+2}{2}$ is folklore, and certainly not new. What is new here is that Theorems \ref{main thm 1} and \ref{main thm 2} identify the input for \eqref{ass 1} in terms of Lie algebra cohomology.)
\end{proof}

We report one more corollary of Theorems \ref{main thm 1} and \ref{main thm 2}:
\begin{corollary}\label{new ss cor}
Let $n$ be a positive integer, and let $p$ be a prime satisfying $p>N_n$. Suppose that $X$ is an $E(n-1)$-acyclic finite CW-complex. Then there exist strongly convergent spectral sequences
\begin{align}
\label{new ss 1} E_1^{*,*,*} \cong H^*(L(n,n);\mathbb{F}_p)\otimes_{\mathbb{F}_p} E_0E(n)_*(X)
  &\Rightarrow \Cotor_{E(n)_*E(n)}^{*,*}(E(n)_*,E(n)_*(X)), \\
\label{En-ass 1} E_2^{*,*} \cong \Cotor_{E(n)_*E(n)}^{*,*}(E(n)_*,E(n)_*(X))
  &\Rightarrow \pi_*(L_{E(n)}X),
\end{align}
where $E_0E(n)_*(X)$ is the associated graded abelian group of the $p$-adic filtration on $E(n)_*(X)$. 
\end{corollary}
\begin{proof}
Spectral sequence \eqref{En-ass 1} is, of course, not new: it is simply the $E(n)$-Adams spectral sequence for $X$, which converges strongly to $\pi_*(L_{E(n)}X)$ by the existence of a horizontal vanishing line at a finite page in the $E(n)$-Adams spectral sequence of every finite CW-complex; see \cite[Theorem A.6.11]{MR1192553}. If $p>n+1$, then the category of $E(n)_*E(n)$-comodules has relative-injective dimension $n^2+n$, so this vanishing line occurs already at the $E_2$-page.

Recall that $\Sigma(n)$ is the Hopf $\mathbb{F}_p$-algebra 
\begin{align*}
 \Sigma(n) &= K(n)_*\otimes_{BP_*}BP_*BP\otimes_{BP_*}K(n)_* \\
 &= E(n)_*E(n)/(p,v_1, \dots ,v_{n-1}, \eta_R(v_1), \dots ,\eta_R(v_{n-1})) \\
 &= E(n)_*E(n)/(p,v_1, \dots ,v_{n-1}).
\end{align*}
Spectral sequence \eqref{new ss 1} arises from a four-step process:
\begin{itemize}
\item First, filter $E(n)_*(X)$ by powers of the ideal $(p, v_1, \dots ,v_{n-1})$ in $E(n)_*$, so that the associated graded $E(n)_*E(n)$-comodule is a $\Sigma(n)$-comodule. This filtration is finite because we have assumed that $X$ is finite and $E(n-1)$-acyclic.
\item Recall (e.g. from the beginning of section 6.2 of \cite{MR860042}) that $\mathbb{Z}$-graded $\Sigma(n)$-comodules are equivalent to $\mathbb{Z}/2(p^n-1)\mathbb{Z}$-graded $S(n,p)$-comodules, via the isomorphism $\Sigma(n)/(1-v_n) \cong S(n,p)$. 
\item Now observe that the pro-$p$-group scheme represented by $\Spec S(n,p)$---i.e., the height $n$ strict Morava stabilizer group scheme---is pro-unipotent. To see this, recall the presentation \eqref{snp pres 1} for $S(n,p)$. For a positive integer $m$, let $S(n,p)_{\leq m}$ be the Hopf subalgebra $\mathbb{F}_p[t_1, \dots ,t_m]/(t_i^{p^n}-t_i\ \forall i)$ of $S(n,p)$. Then $\Spec S(n,p)_{\leq m}$ is the strict automorphism group scheme of the formal $m$-bud truncation of the formal group law $\mathbb{G}_{1/n}$. The affine group scheme $\Spec S(n,p)$ is the limit, over $m$, of the affine group schemes $\Spec S(n,p)_{\leq m}$, so our task is to show that $\Spec S(n,p)_{\leq m}$ is unipotent for each $m$.

We use the filtration criterion for unipotence: see \cite[section 8.3]{MR547117} for an excellent textbook reference. The filtration criterion states that, if $k$ is a field and $A$ is a finitely generated commutative Hopf $k$-algebra, then $\Spec A$ is unipotent if and only if there is a chain of $k$-linear subspaces $A_0 \subseteq A_1 \subseteq A_2 \subseteq \dots$ with $A_0 =k$ and with $\Delta(x) \in \sum_{i=0}^j A_i\otimes_k A_{j-i}$ for all $x\in A_j$. This condition is satisfied by $S(n,p)$, hence also\footnote{Both $S(n,p)$ and $S(n,p)_{\leq m}$ satisfy the filtration criterion for unipotence. We introduce $S(n,p)_{\leq m}$ because it is finitely generated, and as far as we know, unipotence is only studied when an affine group scheme is algebraic, i.e., when the representing Hopf algebra of the group scheme is finitely generated. In particular, the author does not know whether the filtration criterion for unipotence guarantees that the only irreducible comodules are one-dimensional, except when the representing Hopf algebra is finitely generated.} by $S(n,p)_{\leq m}$, by Ravenel's recursive formula \cite[Theorem 4.3.13]{MR860042} for the coproduct in $BP_*BP$ and hence also in $S(n,p)_{\leq m}$. 
\item Since $\Spec S(n,p)$ is pro-unipotent, the only irreducible continuous representations of $\Spec S(n,p)$ are one-dimensional (e.g. see exercises 5-6 in chapter 9 of \cite{MR547117}). Hence the $(p, \dots ,v_{n-1})$-adic filtration on $E(n)_*(X)/(1-v_n)$ admits a finite refinement by subcomodules whose filtration quotients are each one-dimensional. The relationship between $\Sigma(n)$-comodules and $S(n,p)$-comodules lets us lift this filtration to a finite filtration of $E(n)_*(X)$ by $E(n)_*E(n)$-subcomodules whose filtration quotients are each copies of the trivial $\Sigma(n)$-comodule $K(n)_*$.

Consequently the spectral sequence arising from applying\linebreak $\Cotor_{E(n)_*E(n)}^{*,*}(E(n)_*,-)$ to this filtration has $E_1$-page 
\begin{align}
\label{iso a1} \Cotor_{E(n)_*E(n)}^{*,*}(E(n)_*,E_0E(n)_*(X))
  &\cong  \Cotor_{\Sigma(n)}^{*,*}(K(n)_*,E_0E(n)_*(X)) \\
\nonumber
  &\cong  \Cotor_{\Sigma(n)}^{*,*}(K(n)_*,K(n)_*)\otimes_{K(n)_*} E_0E(n)_*(X) \\
\nonumber
  &\cong  \left(H^*(L(n,n);\mathbb{F}_p) \otimes_{\mathbb{F}_p} K(n)_*\right)\otimes_{K(n)_*} E_0E(n)_*(X) \\
\nonumber  &\cong  H^*(L(n,n);\mathbb{F}_p) \otimes_{\mathbb{F}_p} E_0E(n)_*(X) ,
\end{align}
where \eqref{iso a1} is from the Morava-Miller-Ravenel change-of-rings isomorphism. 
Here $E_0E(n)_*(X)$ is the associated graded of the finite filtration we have constructed on $E(n)_*(X)$ using pro-unipotence of $\Spec S(n,p)$, but as an abelian group, it agrees with the associated graded of the $p$-adic filtration on $E(n)_*(X)$. Strong convergence of the spectral sequence is a routine consequence of finiteness of the filtration.
\end{itemize}
\end{proof}

\begin{remark}\label{Nn remark}
The proof of Lemma \ref{ss map is surj} is nonconstructive: given a positive integer $n$, there is no obvious way to write down an integer $M$ such that the left-hand vertical map in \eqref{ss map 1} is surjective for all primes $p>M$. However, all existing calculations are consistent with the possibility that $M$ can be chosen to be simply $1$: that is, the left-hand vertical map in \eqref{ss map 1} is surjective for all $n$ and all $p$. 

Theorems \ref{main thm 1} and \ref{main thm 2} inherit their nonconstructiveness from Lemma \ref{ss map is surj}. For a fixed choice of $n$, let $N_n$ denote the least integer $N\geq n+1$ satisfying the conditions described in the statement of Theorem \ref{main thm 1}. We would very much like to have some understanding of the sequence of integers $N_1, N_2,N_3,\dots$. In \cref{ht 1 example}, we see that $N_1 = 2$, and in \cref{ht 2 example}, we see that $N_2=3$. The ``optimistic answer'' described in \cref{Introduction.} is the proposition that $N_n = n+1$ for all $n\geq 1$. All existing calculations are consistent with the possibility that the RMSS collapses for all $p>n+1$, i.e., $N_n = n+1$, although as Ravenel remarks following Theorem 6.3.5 of \cite{MR860042}, Ravenel suspects that this possibility is dashed by a hypothetical RMSS differential for $n=9, p=11$.
\end{remark}

\section{Explicit example calculations.}
\label{examples section}

The mod $p$ reductions of the integral calculations in this section can be compared with those of section 3 of \cite{ravenel1977cohomology} or section 6.3 of \cite{MR860042}, to see that we indeed recover known cases of the mod $p$ cohomology of Morava stabilizer algebras.

\subsection{Height $1$.}
\label{ht 1 example}

The Lie $\mathbb{Z}$-algebra $\mathcal{Z}(1)$ is the abelian Lie $\mathbb{Z}$-algebra on a single generator $x_1$. Consequently we have $H^*(\mathcal{Z}(1);\mathbb{Z})\cong \Lambda_{\mathbb{Z}}(h_{10})$. Upon reduction modulo $p$ for $p$ sufficiently large, we recover the cohomology $\Lambda_{\mathbb{F}_p}(h_{10})$ of the height $1$ strict Morava stabilizer group; in fact, $p>2$ suffices.

\subsection{Height $2$.}
\label{ht 2 example}

The Lie $\mathbb{Z}$-algebra $\mathcal{Z}(2)$ has $\mathbb{Z}$-linear basis $x_{10},x_{11},x_{20},x_{21}$, with Lie bracket determined by $[x_{10},x_{11}] = x_{20} - x_{21}$ and by $x_{20},x_{21}$ being central in $\mathcal{Z}(2)$. Running the Cartan-Serre spectral sequence for the extension of Lie $\mathbb{Z}$-algebras
\[ 1 \rightarrow \mathbb{Z}\{ x_{20},x_{21}\} \rightarrow \mathcal{Z}(2) \rightarrow \mathbb{Z}\{ x_{10},x_{11}\} \rightarrow 1,\]
one calculates easily that 
\begin{align*}
 H^*(\mathcal{Z}(2);\mathbb{Z})
 &\cong \Lambda_{\mathbb{Z}}(\zeta_2) \otimes_{\mathbb{Z}}\left(\Lambda_{\mathbb{Z}}(h_{10},h_{11})\otimes_{\mathbb{Z}}\mathbb{Z}[h_{10}h_{20},h_{11}h_{20}]
\right)/ \\ & \ \ \ \ \ \ \ \left(h_{10}\cdot h_{11}, (h_{10}h_{20})^2, (h_{11}h_{20})^2, \right. \\ 
 & \ \ \ \ \ \ \ \ \ \ \ \ \ \ \left. (h_{10}h_{20})\cdot h_{11} = -(h_{11}h_{20})\cdot h_{10}, h_{10}h_{20}\cdot h_{11}h_{20}\right),
\end{align*} where $\zeta_2$ is the cohomology class of the Chevalley-Eilenberg cocycle $h_{20} + h_{21}$. The action of $C_2$ is by 
\begin{align*}
 \sigma(h_{10}) &= h_{11} &\ \ \ \ \ \ \ \sigma(h_{11}) &= h_{10} \\
 \sigma(h_{10}h_{20}) &= h_{11}h_{20} - h_{11}\zeta_2 &\ \ \ \ \ \ \ \sigma(h_{11}h_{20}) &= h_{10}h_{20} - h_{10}\zeta_2 \\
 \sigma(\zeta_2) &= \zeta_2. &&
\end{align*}
Upon reduction modulo $p$ for $p$ sufficiently large, we recover the cohomology of the height $2$ strict Morava stabilizer group (compare \cite[Theorem 3.2]{ravenel1977cohomology} or \cite[Theorem 6.3.22]{MR860042}); in fact, $p>3$ suffices.

\appendix 
\section{Simultaneous integral lift of the height $n$ Morava stabilizer algebras at all primes.}
\label{A simultaneous integral lift...}

In this appendix, for each positive integer $n$, we define a differential graded $\mathbb{Z}$-algebra $S_{int}(n)^{\bullet}$ which, in a sense that we make precise, ``specializes'' at each prime $p$ to the $p$-primary height $n$ Morava stabilizer algebra $S(n,p)$. We do not need $S_{int}(n)^{\bullet}$ for the main theorems in this paper, because the smaller object $\mathcal{Z}(n)$ suffices for our purposes. This is because the main theorems in this paper are asymptotic: they are about the cohomology of the $p$-primary height $n$ strict Morava stabilizer groups for $p>>n$. For large $p$, the Hopf algebra $\mathcal{Z}(n)$ is an integral lift of a large enough piece of the Morava stabilizer algebras $S(n,p)$ to detect all the cohomology of $S(n,p)$ at large primes $p$. Readers with any interest in small-primary phenomena will perhaps want to know about an integral lift of the height $n$ Morava stabilizer algebras which ``gets the cohomology right'' even at small primes. That integral lift is $S_{int}(n)^{\bullet}$.

\subsection{The definition of the integral lift $S_{int}(n)^{\bullet}$.}
\label{The integral lift...}

It would be nice to have a single Hopf $\mathbb{Z}$-algebra $S_{int}(n)$ whose mod $p$ reduction, for each $p$, is isomorphic to $S(n,p)$. We are skeptical that such a Hopf $\mathbb{Z}$-algebra can exist. Instead we will construct a differential graded $\mathbb{Z}$-algebra $S_{int}(n)^{\bullet}$ whose reduction at each prime $p$ is the cobar complex of $S(n,p)$. We give the construction in several steps, as follows:
\begin{itemize}
\item Let $R(n)$ denote the free commutative $\mathbb{Z}[C_n]$-algebra on a set of generators $\{ \lowert_1, \lowert_2, \dots\}$. 
\item Let $T(n)^{\bullet}$ denote the differential graded $\mathbb{Z}[C_n]$-algebra whose degree $i$ summand is the $i$-fold tensor power $R(n)^{\tensor_{\mathbb{Z}} i}$ of $R(n)$. We refer to $R(n)^{\tensor_{\mathbb{Z}} i}$ as the group of {\em $i$-cochains} in $T(n)^{\bullet}$. 

We define the $C_n$-action on $T(n)^{\bullet}$ by letting the operator $\sigma$ act diagonally on $R(n)^{\tensor_{\mathbb{Z}} i}$. Hence, for example, $\sigma(\lowert_j\otimes \lowert_k) = (\sigma t_j)\otimes (\sigma t_k)$. 

The graded algebra $T(n)^{\bullet}$ has a certain feature in common with the cobar complex of a Hopf algebra (or Hopf algebroid), as defined in appendix 1 of \cite{MR860042}: it has {\em two} natural-looking multiplication operations, and unless care is taken to distinguish them, there will be confusion and errors. The degree $i$ summand $R(n)^{\tensor_{\mathbb{Z}} i}$ of $T(n)^{\bullet}$ is, in itself, a ring, so we can multiply two $i$-cochains and get an $i$-cochain. We will call this multiplication the {\em internal product} on $T(n)^{\bullet}$. By contrast, we can instead take an element $x\in R(n)^{\tensor_{\mathbb{Z}} i}$ and an element $y\in R(n)^{\tensor_{\mathbb{Z}} j}$, and concatenate them across a tensor symbol to get an element $x\otimes y\in R(n)^{\tensor_{\mathbb{Z}} (i+j)}$. We will call this multiplication the {\em cup product} on $T(n)^{\bullet}$. For the sake of producing a differential graded algebra, it is the cup product on $T(n)^{\bullet}$ which is the relevant multiplication operation.
\item Let $S_{int}(n)^{\bullet}$ be the graded $\mathbb{Z}[C_n]$-algebra obtained from $T(n)^{\bullet}$ by freely (in the category of associative graded $\mathbb{Z}[C_n]$-algebras) adjoining generators $\lowerb_1, \lowerb_2, \lowerb_3, \dots$ in degree $2$.
\item We define a $\mathbb{Z}[C_n]$-linear differential $d: S_{int}(n)^{\bullet}\rightarrow S_{int}(n)^{\bullet}$ of degree $+1$ by the following rules:
\begin{align}
\label{dt formula} d(\lowert_i) &= \sum_{j=1}^{i-1} \lowert_j \otimes \sigma^jt_{i-j} + \sum_{k=1}^{i/n} \sigma^{kn+1} b_{i-kn}, \\
\label{db formula} d(\lowerb_i) &= -\sigma^{-1-kn}\left( \sum_{j=1}^{i+n-1} d(\lowert_j\otimes \sigma^j t_{i+n-j}) + \sum_{k=1}^{\frac{i}{n}-1} d(\lowerb_{i-nk})\right)
\end{align}
with the understanding that $\lowerb_i$ is zero if $i\leq 0$.

The formula \eqref{db formula} is not a closed-form formula for $d(\lowerb_i)$, but it is enough to let us solve recursively for $d(\lowerb_i)$ by first calculating $d(\lowerb_1),d(\lowerb_2)$, and so on. 
\end{itemize}
Note that \eqref{db formula} is a consequence of \eqref{dt formula}: since $d\circ d$ must be zero, apply $d$ to both sides of \eqref{dt formula} and solve for $d(\lowerb_{i-n})$, then re-index to arrive at formula \eqref{db formula}. 
\begin{definition}\label{def of int lift alg}
For each positive integer $n$, the {\em integral lift of the height $n$ Morava stabilizer algebra} is the $C_n$-equivariant differential graded $\mathbb{Z}$-algebra $S_{int}(n)^{\bullet}$ defined above. 
\end{definition}

The cobar complex of $\mathcal{Z}(n)$ embeds into $S_{int}(n)^{\bullet}$ by the differential graded $\mathbb{Z}[C_n]$-algebra homomorphism $\mathcal{Z}(n)\rightarrow S_{int}(n)^{\bullet}$ sending $\lowert_i$ to $\lowert_i$ and preserving both the internal product and the cup product.


\subsection{How to specialize $S_{int}(n)^{\bullet}$ to a prime.}
\label{Specialization of...}

We now need to explain the sense in which $S_{int}(n)^{\bullet}$ is an integral lift of the Morava stabilizer algebras $S(n,2)$, $S(n,3)$, $S(n,5),\dots$. It is not a matter of simply reducing the $\mathbb{Z}$-algebra $S_{int}(n)^{\bullet}$ modulo a prime $p$ to recover $S(n,p)$. Instead, to recover the cobar complex of $S(n,p)$, one must ``specialize'' $S_{int}(n)^{\bullet}$ in the following way:
\begin{definition}
Fix a positive integer $n$ and a prime $p$. 
By the {\em specialization of $S_{int}(n)^{\bullet}$ at $p$,} we mean the $C_n$-equivariant differential graded $\mathbb{F}_p$-algebra obtained by:
\begin{itemize}
\item setting $p$ to zero,
\item setting $\sigma$ to the internal Frobenius endomorphism, i.e., setting $\sigma(x) = x^p$ for all homogeneous $x$, {\em where the $p$th power is taken using the internal product}, as defined in \cref{The integral lift...},
\item and, for each positive integer $i$, setting $\lowerb_i$ to $-1$ times the transpotent of $\lowert_i$,
i.e., setting $\lowerb_i$ to the mod $p$ reduction of the integral sum $-\frac{1}{p} \sum_{j=1}^{p-1}\binom{p}{j} \lowert_i^j\otimes \lowert_i^{p-j}$.
\end{itemize}
We will write $S_{int}(n)^{\bullet}\downarrow p$ for the specialization of $S_{int}(n)^{\bullet}$ at $p$.
\end{definition}
Write $C_{S(n,p)}^{\bullet}$ for the cobar complex of $S(n,p)$. Then $C_{S(n,p)}^{\bullet}$ is a $C_n$-equivariant differential graded $\mathbb{F}_p$-algebra, with $C_n$ acting by the internal Frobenius map $x \mapsto x^p$, i.e., the zeroth algebraic Steenrod operation $P^0$, as defined in \cite[appendix 1 section 5]{MR860042}. 
It is straightforward to use the calculations of sections 4.3 and 6.3 of \cite{MR860042} to verify that the map of $C_n$-equivariant differential graded algebras 
\begin{align*}
 S_{int}(n)^{\bullet}\downarrow p &\rightarrow C_{S(n,p)}^{\bullet} \\
 \lowert_i &\mapsto t_i \\
 \lowerb_i &\mapsto b_{i,0} \end{align*}
is an isomorphism.


\subsection{The $q$-Ravenel filtration on $S_{int}(n)^{\bullet}$.}
\label{The Ravenel filtration on...}

To show that the Ravenel-May spectral sequence also has an integral lift, even at small primes, we define a lift of the Ravenel filtration to the integral lift $S_{int}(n)^{\bullet}$ of the Morava stabilizer algebras. In fact we will have a family of such lifts, one lift for each integer $q\geq 2$. It is defined in the same way as the $q$-Ravenel filtration on $\mathcal{Z}(n)$, from Definition \ref{def of rav filt on Zn}:

\begin{definition}
For integers $n\geq 1$ and $q\geq 2$, let $d_{n,q,i}$ be the integer defined recursively by the rule 
\begin{align*}
 d_{n,q,i} &= \left\{ \begin{array}{ll} 0 &\mbox{\ if\ } i\leq 0 \\ \max\{ i,qd_{n,q,i-n}\} &\mbox{\ if\ } i>0.\end{array}\right.
\end{align*}
Now equip the $C_n$-equivariant differential graded $\mathbb{Z}$-algebra $S_{int}(n)^{\bullet}$ with the increasing filtration in which, for all $j$, the element $\sigma^jt_i$ is in degree $d_{n,q,i}$ and the element $\sigma^jb_i$ is in degree $d_{n,q,i+n}$.
We call this the {\em $q$-Ravenel filtration on $S_{int}(n)^{\bullet}$.}
\end{definition}

In the case $p=q$, the $q$-Ravenel filtration on $S_{int}(n)^{\bullet}$ specializes (in the sense of \cref{Specialization of...}) to the Ravenel filtration on $S(n,p)$.

\def\cprime{$'$} \def\cprime{$'$} \def\cprime{$'$} \def\cprime{$'$}


\begin{thebibliography}{10}

\bibitem{MR0024908}
Claude Chevalley and Samuel Eilenberg.
\newblock Cohomology theory of {L}ie groups and {L}ie algebras.
\newblock {\em Trans. Amer. Math. Soc.}, 63:85--124, 1948.

\bibitem{MR2030586}
Ethan~S. Devinatz and Michael~J. Hopkins.
\newblock Homotopy fixed point spectra for closed subgroups of the {M}orava
  stabilizer groups.
\newblock {\em Topology}, 43(1):1--47, 2004.

\bibitem{MR4301320}
Xing Gu, Xiangjun Wang, and Jianqiu Wu.
\newblock The composition of {R}. {C}ohen's elements and the third periodic
  elements in stable homotopy groups of spheres.
\newblock {\em Osaka J. Math.}, 58(2):367--382, 2021.

\bibitem{MR0157972}
Ju.~I. Manin.
\newblock Theory of commutative formal groups over fields of finite
  characteristic.
\newblock {\em Uspehi Mat. Nauk}, 18(6 (114)):3--90, 1963.

\bibitem{MR0193126}
J.~P. May.
\newblock The cohomology of restricted {L}ie algebras and of {H}opf algebras.
\newblock {\em J. Algebra}, 3:123--146, 1966.

\bibitem{MR2614527}
J.~Peter May.
\newblock {\em T{he} {cohomology} {of} {restricted} {Lie} {algebras} {and} {of}
  {Hopf} {algebras}: {application} {to} {the} {Steenrod} {algebra}}.
\newblock ProQuest LLC, Ann Arbor, MI, 1964.
\newblock Thesis (Ph.D.)--Princeton University.

\bibitem{MR0185595}
J.~Peter May.
\newblock The cohomology of restricted {L}ie algebras and of {H}opf algebras.
\newblock {\em Bull. Amer. Math. Soc.}, 71:372--377, 1965.

\bibitem{MR0458410}
Haynes~R. Miller and Douglas~C. Ravenel.
\newblock Morava stabilizer algebras and the localization of {N}ovikov's
  {$E_{2}$}-term.
\newblock {\em Duke Math. J.}, 44(2):433--447, 1977.

\bibitem{MR0174052}
John~W. Milnor and John~C. Moore.
\newblock On the structure of {H}opf algebras.
\newblock {\em Ann. of Math. (2)}, 81:211--264, 1965.

\bibitem{MR782555}
Jack Morava.
\newblock Noetherian localisations of categories of cobordism comodules.
\newblock {\em Ann. of Math. (2)}, 121(1):1--39, 1985.

\bibitem{MR0420619}
Douglas~C. Ravenel.
\newblock The structure of {M}orava stabilizer algebras.
\newblock {\em Invent. Math.}, 37(2):109--120, 1976.

\bibitem{ravenel1977cohomology}
Douglas~C Ravenel.
\newblock The cohomology of the {M}orava stabilizer algebras.
\newblock {\em Mathematische Zeitschrift}, 152(3):287--297, 1977.

\bibitem{MR860042}
Douglas~C. Ravenel.
\newblock {\em Complex cobordism and stable homotopy groups of spheres}, volume
  121 of {\em Pure and Applied Mathematics}.
\newblock Academic Press Inc., Orlando, FL, 1986.

\bibitem{MR1192553}
Douglas~C. Ravenel.
\newblock {\em Nilpotence and periodicity in stable homotopy theory}, volume
  128 of {\em Annals of Mathematics Studies}.
\newblock Princeton University Press, Princeton, NJ, 1992.
\newblock Appendix C by Jeff Smith.

\bibitem{height4}
A.~Salch.
\newblock The cohomology of the height $4$ {M}orava stabilizer group at large
  primes.
\newblock {\em draft version available on arXiv}, 2016.

\bibitem{MR547117}
William~C. Waterhouse.
\newblock {\em Introduction to affine group schemes}, volume~66 of {\em
  Graduate Texts in Mathematics}.
\newblock Springer-Verlag, New York, 1979.

\bibitem{MR1173994}
Atsushi Yamaguchi.
\newblock The structure of the cohomology of {M}orava stabilizer algebra
  {$S(3)$}.
\newblock {\em Osaka J. Math.}, 29(2):347--359, 1992.

\end{thebibliography}
\end{document}